\documentclass{amsart}
\usepackage{graphicx}
\usepackage{amsmath}
\usepackage{amsfonts}
\usepackage{amssymb}
\usepackage{mathtools}
\usepackage{tikz-cd} 	
\usepackage{parskip}
\usepackage{amsthm}
\usepackage{appendix}
\usepackage{tikz}
\usetikzlibrary{calc}

\DeclareMathOperator{\ima}{im}
\DeclareMathOperator{\coker}{coker}
\DeclareMathOperator{\id}{id}

\DeclareMathOperator{\Hilb}{Hilb}
\DeclareMathOperator{\Rays}{Rays}

\DeclareMathOperator{\Spec}{Spec}

\DeclareMathOperator{\cO}{\mathcal{O}}

\DeclareMathOperator{\RR}{\mathbf{R}}
\DeclareMathOperator{\ZZ}{\mathbf{Z}}

\DeclareMathOperator{\QQ}{\mathbf{Q}}

\DeclareMathOperator{\dual}{\!^{\vee}}

\DeclareMathOperator{\dani}{\widetilde{\Omega}}

\DeclareMathOperator{\tors}{tors}

\DeclareMathOperator{\supp}{supp}
\DeclareMathOperator{\rank}{rank}

\DeclareMathOperator{\cdh}{cdh}
\DeclareMathOperator{\zar}{zar}

\DeclareMathOperator{\textif}{ if }
\newcommand{\textst}{\text{ such that }}

\newtheorem{theorem}{Theorem}[section]

\newtheorem{lemma}[theorem]{Lemma}

\theoremstyle{definition}
\newtheorem{definition}[theorem]{Definition}

\newtheorem*{remark}{Remark}

\newtheorem{claim}[theorem]{Claim}

\RequirePackage{tikz-cd}
\RequirePackage{amssymb}
\usetikzlibrary{calc}
\usetikzlibrary{decorations.pathmorphing}

\tikzset{curve/.style={settings={#1},to path={(\tikztostart)
    .. controls ($(\tikztostart)!\pv{pos}!(\tikztotarget)!\pv{height}!270:(\tikztotarget)$)
    and ($(\tikztostart)!1-\pv{pos}!(\tikztotarget)!\pv{height}!270:(\tikztotarget)$)
    .. (\tikztotarget)\tikztonodes}},
    settings/.code={\tikzset{quiver/.cd,#1}
        \def\pv##1{\pgfkeysvalueof{/tikz/quiver/##1}}},
    quiver/.cd,pos/.initial=0.35,height/.initial=0}

\tikzset{tail reversed/.code={\pgfsetarrowsstart{tikzcd to}}}
\tikzset{2tail/.code={\pgfsetarrowsstart{Implies[reversed]}}}
\tikzset{2tail reversed/.code={\pgfsetarrowsstart{Implies}}}
\tikzset{no body/.style={/tikz/dash pattern=on 0 off 1mm}}

\title{$K_2$-regularity is Equivalent to Smoothness for Toric Varieties}
\author{Christian Haesemeyer}
\address{School of Mathematics and Statistics, University of Melbourne,
VIC 3010, Australia}
\email{christian.haesemeyer@unimelb.edu.au}

\author{George Henderson-Walshe}
\address{Dipartimento di Matematica, Università di Bologna, Piazza di Porta San Donato 5 - 40126, Bologna, Italy}
\email{george.henderson@studio.unibo.it}
\date{\today}

\begin{document}

\begin{abstract}
    Recent work by W.\,Shen \cite{Shen} shows that $K_2$-regular affine local complete intersections of any dimension in characteristic zero are smooth. In this note, we show that the same holds for affine toric varieties. 
\end{abstract}

\maketitle

\section{Introduction}

Over the past twenty years, there has been major progress on long-standing questions regarding the structure of algebraic $K$-theory of (algebraic) singularities. It has been known since the early days of the subject that the $K$-theory of singularities will not necessarily be algebraically homotopy invariant: for a singular scheme $X$, and an integer $n$, the natural map $K_n(X)\to K_n(X\times\mathbb{A}^1)$ will not typically be an isomorphism. Bass introduced the terminology that $X$ is $K_n$-regular if the above map is an isomorphism for all $m\leq n$.

In \cite{Vorst1}, Ton Vorst analysed the case of one-dimensional singularities essentially completely: if $X$ is an affine curve of finite type over a field, then $X$ is $K_2$-regular if, and only if, it is regular. Based on this result and its proof, Vorst conjectured that for a $d$-dimensional algebra $R$ of finite type over a field $F$, $K_{d+1}$-regularity implies regularity. This conjecture was proved in \cite{chw} when $F$ has characteristic zero, and for a slightly weaker notion of dimension (that coincides with Krull dimension for algebras of finite type over perfect fields), in \cite{KST21} in positive characteristic. 

In characteristic zero, the proof proceeds via trace methods: the Jones--Goodwillie Chern character can be used to calculate $NK_n(R) = \mathrm{coker}\left(K_n(R)\to K_n(R\times\mathbb{A}^1) \right)$ in terms of Hochschild homology and its sheafification in the $cdh$-topology (see \cite{Bass_NK_groups}), and this in turn can be used to show that $K_{d+1}$-regularity of $R$ implies vanishing of the $(d+1)$-st K\"ahler differentials, which means $R$ is a smooth $F$-algebra by the Jacobi criterion. 

Somewhat surprisingly, there are currently no examples in the literature demonstrating that Vorst's Conjecture is sharp in any dimension above one. Indeed, recent work by Wanchun Shen (\cite{Shen}) using previous results on the Hodge structures of local cohomology has shown that, when $R$ is local complete intersection, $K_2$-regularity suffices to imply regularity. In this note, we show that the same is true when $R$ has toric singularities: 

\begin{theorem}\label{thm:K1-intro}
    Let $X = U(\sigma, N)$ be an affine toric variety over a field of characteristic zero. If $\sigma$ is simplicial, then $X$ is smooth if, and only if, $X$ is $K_1$-regular. 
\end{theorem}

\begin{theorem}\label{thm:K2-intro}
    Let $X = U(\sigma, N)$ be an affine toric variety over a field of characteristic zero. If $\sigma$ is not simplicial, then $X$ is not $K_2$-regular. 
\end{theorem}

The proof of these results proceeds by combinatorial calculation and relies on three inputs: The explicit description of the groups $NK_1$ and $NK_2$ in \cite[Theorem 5.1]{Bass_NK_groups}; the identification (contained in \cite{chww-toric}) of the $cdh$-sheaves of differentials on toric varieties with the sheaves Danilov defines in \cite[Section 4]{danilov_geometry_1978}; and the description of the $M$-graded structure on these sheaves (see \cite[(4.2.3)]{danilov_geometry_1978} and \cite[(3.3)]{chww-toric}). 

We provide a short outline of the paper. In section \ref{sec:preliminary}, we very briefly review the necessary $K$-theoretic and toric geometric background and identify the obstructions to $K$-regularity that we will use later. In section \ref{sec:structural}, we prove several lemmas that help us to compute our obstruction groups via reduction to easier cases, in particular the case of a full-dimensional cone. In section \ref{sec:simplicial} we deal with the simplicial case, proving Theorem \ref{thm:K1-intro} (Theorem \ref{thm:K1}). Finally, in section \ref{sec:non-simplicial} we proceed by induction on the dimension of the toric variety to prove Theorem \ref{thm:K2-intro} (Theorem \ref{thm:K2}); the crucial step is the base case (dimension $3$), which we treat in Lemma \ref{lemma: base case}. 

Throughout the present paper, $k$ denotes a field of characteristic $0$, and all toric varieties will be assumed to be normal varieties over $k$. Also, all cones are assumed to be rational and polyhedral.

\subsection*{Acknowledgments:} This paper is based on the second author's Masters thesis at the University of Melbourne, supervised by the first author.  

\subsection*{Statement on AI use:} No AI was used in the preparation of this paper, whose mathematics was mostly completed before the recent rapid advances in AI capability. 

\section{Preliminaries and Notation}\label{sec:preliminary}

Let $N$ be a lattice and $M$ its dual lattice, whose pairing $M \times N \to \ZZ$ is written $\langle m, n \rangle$. We denote the corresponding real vector spaces by $N_{\RR}\coloneq N \otimes \RR$ and $M_{\RR} \coloneq M \otimes \RR$, respectively. We use the letter $\sigma$ to denote a strongly convex cone in $N$ and the letters $\rho$ and $\tau$ for faces of $\sigma$, written $\rho < \sigma$, $\tau < \sigma$. The dual of a cone $\sigma$ is written $\sigma\dual$, and the letters $\eta$ and $\theta$ are used for faces of $\sigma\dual$.

We denote by $U(\sigma,N) = \Spec (k[S_\sigma])$ the affine toric variety constructed from a cone $\sigma$ in a lattice $N$, where $S_\sigma$ denotes the monoid $\sigma\dual \cap M$. The monoid algebra $k[S_\sigma]$ is spanned by the monomials $\chi^m$, $m \in S_\sigma$. The toric variety constructed from a fan $\Delta$ in a lattice $N$ is denoted $X(\Delta,N)$. For a face $\rho$ in a fan $\Delta$, we denote by $i_\rho :V_\Delta(\rho) \to X(\Delta,N)$ the closed immersion of the associated orbit closure toric variety (c.f. \cite[Section 3.1]{fulton_introduction_1993}). Finally, for an integer $r$ we denote by $\Delta(r)$ the set of $r$-dimensional cones in a fan $\Delta$, and by $C(r)$ the set of $r$-dimensional faces of a cone $C$.

For a toric variety $X=X(\Delta,N)$ and an integer $q$, we denote by $\Omega^q_X$ the sheaf of degree $q$ K\"ahler differentials over $k$ on $X$. An $M$-grading on $\Omega_X^q$ is given by setting $\deg(\chi^{m_0} d\chi^{m_1} \wedge \cdots \wedge d\chi^{m_q}) = \sum_{i=0}^q m_i$. For toric varieties, there is a related sheaf, introduced by Danilov \cite[4.2]{danilov_geometry_1978}, which can be described as follows.

\begin{definition}\cite[Definition 3.1]{chww-toric}
    For a toric variety $X = X(\Delta,N)$, we define $\dani_X^q$ to be the $M$-graded coherent sheaf fitting into the sequence 
\begin{equation} \label{eq:exact_1}
\begin{tikzcd}
	0 & {\dani_{X}^q } & {\cO_{X} \otimes \bigwedge^q(M)} & {\bigoplus_{\rho \in \Delta(1)} \cO_{V_\Delta(\rho)} \otimes \bigwedge^{q-1}(\rho^\perp \cap M)}
	\arrow[from=1-1, to=1-2]
	\arrow[from=1-2, to=1-3]
	\arrow[from=1-3, to=1-4, "\delta"]
\end{tikzcd}.
\end{equation}
The component of the map $\delta$ indexed by $\rho$ sends $f \otimes (m_1 \wedge \cdots \wedge m_q)$ in $\cO_X \otimes \bigwedge^q(M)$ to 
\begin{equation}
    i_\rho^{*}(f) \otimes \left( \sum_{i=1}^q \langle m_i, n_\rho \rangle m_1 \wedge \cdots \wedge \hat{m_i} \wedge \cdots \wedge m_q  \right),
\end{equation}
where $n_\rho$ denotes the generator of $\rho \cap N$. 
\end{definition}

 For a cone $\sigma$ in $N$, we denote by $N_\sigma$ the smallest sublattice of $N$ such that $\sigma \subset N_{\RR}$, and by $V_{\sigma}$ the vector space $N_\sigma \otimes k$. For a lattice point $m \in M $, we define the face $\Gamma(m) \in \Delta$ to be the smallest face in $\Delta$ which contains $m$ (if such a cone exists). 
\begin{lemma}\cite[Equation 3.3]{chww-toric}
    For an affine toric variety $X=U(\sigma,N)$ and an element $m \in M$ Danilov's sheaf is given explicitly in degree $m$ by
    \begin{equation}
        \dani^q_{X}(m) = \begin{cases}
            \chi^m \otimes \bigwedge^q \left(V_{\Gamma(m)} \right) \textif m \in \sigma\dual \\
            0 \textif m \notin \sigma\dual.
            \end{cases}
    \end{equation}
\end{lemma}

In other words, $\dani_X^q(m)$ can be described only in terms of $\Gamma(m)$, ignoring more complicated structure in $S_\sigma$. This is also true for $\Omega_X^q$ when $X$ is smooth, but not in general. We have: 
\begin{lemma}\cite[Proposition 4.3]{danilov_geometry_1978} \label{lemma: alpha}
    If $X = X(\Delta,N)$ is a smooth toric variety then there is a natural isomorphism $\alpha^q_X: \dani_{ X}^q \to \Omega_X^q$ defined by
    \begin{equation}
        \chi^{m_0} \otimes \left( m_1 \wedge \cdots \wedge m_q \right) \mapsto \chi^m d\chi^{m_1} \wedge \cdots \wedge d\chi^{m_q},
    \end{equation}
    where $m = m_0 - \sum_{i=1}^q m_i$.
\end{lemma}
In the $cdh$ topology, blow-ups give covers. For toric varieties, a \emph{refinement} of the associated fan corresponds to an abstract blow-up. Any fan can be refined to a smooth fan in finitely many steps, and for any two refinements of the same fan, there must exist a finite common refinement \footnote{Recent work has proven the stronger result that for any two refinements, there exists a common refinement by stellar subdivisions. Stellar subdivisions correspond to actual toric blow-ups, not just abstract blow-ups \cite{adiprasito2024triangulationscommonstellarsubdivision}.}. Consequently, every toric variety has a $cdh$ cover by smooth toric varieties. The authors of \cite{chww-toric} proved that $\dani_X^q$ is a $cdh$ sheaf, and therefore the natural isomorphisms $\alpha$ of Lemma \ref{lemma: alpha} induce an isomorphism in $cdh$ cohomology groups. Let $(a_X^q)_*:H^p_{\zar}(X,\Omega^q_X) \to H^p_{\cdh}(X,\Omega_X^q)$ and $(\tilde a_X^q)_*: H^p_{\zar}(X,\dani_X^q) \to H^p_{\cdh}(X,\dani_X^q)$ be the maps induced by sheafification. 

\begin{lemma} \cite[Theorem 4.1]{chww-toric} \label{lemma: danilov cdh-sheafification}
    If $X$ is a toric variety then, firstly, the map $(\tilde a_X^q)_*:H^p_{\zar}(X,\dani_X^q) \to H^p_{\cdh}(X,\dani_X^q)$ induced by sheafification is an isomorphism, and secondly, $\alpha^q_X$ induces an isomorphism of cohomology groups
    \begin{equation}
        (\alpha_X^q)_*: H^p_{\cdh}(X,\dani_X^q) \to H^p_{\cdh}(X,\Omega^q_X).
    \end{equation}
\end{lemma}
While Danilov's map $\alpha^q_X$ is defined only for smooth toric varieties, its inverse, which we denote by $\beta^q_X$, can be extended to all toric varieties, making it a natural comparison map.
\begin{lemma} \cite[Equation 3.4]{chww-toric} \label{lemma: beta}
    For a toric variety $X=X(\Delta,N)$, there is a $M$-graded morphism of sheaves $\beta_X^q: \Omega_X^q \to \dani_X^q$ given by
    \begin{equation}
        \chi^{m_0} d\chi^{m_1} \wedge \cdots \wedge d\chi^{m_q} \mapsto \chi^m \otimes \left( m_1 \wedge \cdots \wedge m_q \right),
    \end{equation}
    where $m = \sum_{i=0}^q m_i$. Moreover, if $X$ is smooth then $\beta_X^q$ and $\alpha_X^q$ are inverses.
\end{lemma}
This comparison map is particularly useful because it can be studied directly, as we do in the present paper, but it also relates to the sheafification map on cohomology groups, $(a_X^q)_*$, as follows.
\begin{lemma} \label{lemma: sheafification is beta}
    If $X$ is a toric variety, then for each integer $q$, there are isomorphisms $\ker((a_X^q)_*) \cong \ker((\beta_X^q)_*)$ and $\coker((a_X^q)_*) \cong \coker((\beta_X^q)_*)$, where $(\beta_X^q)_*: H^p(X,\Omega_X^q) \to H^p(X,\dani_X^q)$ is the map on cohomology groups induced by $\beta_X^q$ and $(a_X^q)_*$ is as in Lemma \ref{lemma: danilov cdh-sheafification}.
\end{lemma}
\begin{proof}
    By taking the morphism of sheaves $\beta^q_X$ on covers and applying the sheafification functor, we obtain the following commutative diagram,
\[\begin{tikzcd}
	{H^p_{\zar}(X,\Omega_X^q)} & {H^p_{\cdh}(X,\Omega_X^q)} \\
	{H^p_{\zar}(X,\dani_X^q)} & {H^p_{\cdh}(X,\dani_X^q)}
	\arrow["(a^q_X)_*", from=1-1, to=1-2]
	\arrow["(\beta_X^q)_*"', from=1-1, to=2-1]
	\arrow["{(\beta_X^q)_*}", shift left=3, from=1-2, to=2-2]
	\arrow["\cong"{description}, draw=none, from=1-2, to=2-2]
	\arrow["{(\tilde a_X^q)_*}", from=2-1, to=2-2]
	\arrow["\cong"', draw=none, from=2-1, to=2-2]
	\arrow["{(\alpha_X^q)_*}", shift left=3, from=2-2, to=1-2]
\end{tikzcd}.\]
Therefore $(a_X)_* = (\alpha_X^q)_* \circ (\tilde a_X)_* \circ (\beta_X^q)_*$. Since $(\alpha_X^q)_*$ and $(\tilde a_X)_*$ are isomorphisms by Lemma \ref{lemma: danilov cdh-sheafification}, the result follows. 
\end{proof}
This allows us to prove the non-vanishing of $NK$ groups via the following theorem. We note that, in this theorem, all differentials are taken relative to $\QQ$.
\begin{theorem} \label{thm:NK_groups} \cite[Theorem 5.5]{Bass_NK_groups}
For any commutative ring $R$ containing $\QQ$ and any integer $q$,

     \begin{equation}
        NK_q(R) \cong \left[
        \begin{array}{ll} \coker((a_R^q)_*)  \\
        \oplus \ker((a_R^{q-1})_*)  \\
        \oplus AQ \end{array} \right] \otimes (t \QQ[t]),
    \end{equation}
    where $(a_R^q)_*:H^0_{\zar}(\Omega_R^{q}) \to H^0_{\cdh}(\Omega_{R}^{q})$ is the map on cohomology groups induced by sheafification for the cdh topology, and $AQ$ is a summand which is not relevant to the present paper. 
\end{theorem}
In summary, this section has produced the following Lemma.
\begin{lemma} \label{lemma: ker coker implies not K regular}
    Let $X$ be an affine toric variety. If $\coker(\beta_X^q) \neq 0$ then $X$ is not $K_q$-regular, and if $\ker(\beta_X^q) \neq 0$ then $X$ is not $K_{q+1}$-regular.
\end{lemma}
\begin{proof}
    Assume that $\ker(\beta^q_X)\neq 0$ (or $\coker(\beta^q_X)\neq 0$). By Lemma \ref{lemma: sheafification is beta} with $p=0$, $\ker(\beta^q_X) = \ker(a^q_X)\neq 0$ ($\coker(\beta^q_X) = \coker(a^q_X)\neq 0$, respectively). Here differentials are taken relative to $k$. Since the toric variety is base changed from that described by the same fan over $\QQ$, flat base change and the K\"unneth formula for Hochschild homology imply that the corresponding groups relative to $\QQ$ are also non-zero. (Compare the discussion preceding \cite[Cor.\,4.5]{chww-toric}.) Thus the assertion of the lemma follows from Theorem \ref{thm:NK_groups}.
\end{proof}

\section{Structural Lemmas}\label{sec:structural}

In this section we examine how the comparison map interacts with toric open immersions, products and closed immersions. For the sake of generality and consistency, we continue to view the comparison map as a morphism of sheaves, rather than modules, although all of the varieties considered henceforth are indeed affine.

In this section we will consider $\Omega_X^q$ and $\dani_X^q$ as contravariant functors from toric varieties to $k$-vector spaces, with the obvious induced maps. To make this clearer, we write $\Omega^q(X) = \Omega_X^q$ and $\dani^q(X) =\dani_X^q$. In this language, $\beta^q$ is a natural transformation from $\Omega^q$ to $\dani^q$. It is clear that the dashed arrows in the diagram below can be filled uniquely,

\begin{equation} \label{diagram: functoriality}
\begin{tikzcd}
	{\ker(\beta^q)(Y)} & {\Omega^q(Y)} & {\dani^q(Y)} & {\coker(\beta^q)(Y)} \\
	{\ker(\beta^q)(X)} & {\Omega^q(X)} & {\dani^q(X)} & {\coker(\beta^q)(X)}
	\arrow[from=1-1, to=1-2]
	\arrow[dashed, from=1-1, to=2-1]
	\arrow["{\beta^q(Y)}", from=1-2, to=1-3]
	\arrow[from=1-2, to=2-2]
	\arrow[from=1-3, to=1-4]
	\arrow[from=1-3, to=2-3]
	\arrow[dashed, from=1-4, to=2-4]
	\arrow[from=2-1, to=2-2]
	\arrow["{\beta^q(X)}", from=2-2, to=2-3]
	\arrow[from=2-3, to=2-4].
\end{tikzcd}\end{equation}

These maps make the assignments $\ker(\beta^q): X \mapsto \ker(\beta_X^q)$ and $\coker(\beta^q): X \mapsto \coker(\beta_X^q)$ into functors.

\begin{lemma}\label{lemma: closed lifting}
    Let $X = U(\sigma,N)$ where $\sigma$ is a cone in $N$ and let $V = V_\sigma(\rho)$ for some cone $\rho \in \Delta$. Let $q$ be an integer. Then for $F = \ker(\beta^q)$ or $\coker(\beta^q)$, the vector space $F(V)$ is a direct summand of $F(X)$.
\end{lemma}
\begin{proof}
    Let $i\coloneq i_\rho: V \to X$ be the inclusion map, and $p: X \to V$ the projection map, corresponding to the projection $N \twoheadrightarrow N(\rho) \coloneq N / N_\rho$. Note that $p \circ i = \id_V$. By functoriality, 
    \begin{equation}
        F(i) \circ F(p) = F(p \circ i) = F(\id_V) = \id_{F(V)},
    \end{equation}
    so $F(i)$ is a projection, for $F = \ker(\beta^q)$ or $\coker(\beta^q)$. 
\end{proof}

\begin{lemma} \label{lemma: product}
    If $X$ is a product of toric varieties, with $Y$ as a factor, and $q$ is an integer, then for $F = \ker(\beta^q)$ or $\coker(\beta^q)$, the vector space $F(Y)$ is a direct summand of $F(X)$.
\end{lemma}
\begin{proof}
    Since $Y$ is a factor of $X$, there exists an inclusion $i: Y \to X$ and a projection $p: X \to Y$ such that $p \circ i = \id_Y$. By functoriality,  
    \begin{equation}
        F(i) \circ F(p) = F(p \circ i) = F(\id_Y) = \id_{F(Y)},
    \end{equation}
    so $F(i)$ is a projection, for $F = \ker(\beta^q)$ or $\coker(\beta^q)$.
\end{proof}

\begin{lemma} \label{lemma: open lifting}
    Let $X = U(\sigma,N)$ where $\sigma$ is a strongly convex cone in $N$, and let $Y = U(\tau,N_{\tau})$, where $\tau < \sigma$ is a face. Then for $F = \ker(\beta^q)$ or $\coker(\beta^q)$, with any integer $q$, $F(X) = 0$ implies $F(Y) = 0$.
\end{lemma}
\begin{proof}
    The proof proceeds via an intermediate variety, $U \coloneq U(\tau, N)$. There is a standard open immersion $j:U(\tau,N) \to X$ given by localisation at $k[M(\tau)]$, where $M(\tau) \coloneq \tau^\perp \cap M$. Consequently $U(\tau,N) \cong T_{N(\tau)} \times_{k} U(\tau,N_\tau)$, where $T_{N(\tau)}$ denotes the torus with character lattice $N(\tau) \coloneq N / N_\tau$.
    
    Since $U$ and $X$ are affine and  $\ker(\beta_X^q)$ and $\coker(\beta^q_X)$ are coherent sheaves on $X$, $F(U) = k[M(\tau)] \otimes F(X)$ for $F = \ker(\beta^q)$ and $\coker(\beta^q)$. Therefore, if $F(X) = 0 $ then $F(U) = 0$.

    Since $Y = U(\tau, N_\tau)$ is a factor of the product variety $U$, Lemma \ref{lemma: product}, $F(Y)$ is a direct summand of $F(U)$. Hence if $F(X)=0$ then $F(Y)=0$.
\end{proof}

\begin{lemma}\label{lemma: torsion-kernel}
    Let $X = U(\sigma,N)$ be an affine toric variety. Then for each integer $q$,
    \begin{equation}
        \ker(\beta^q)(X) = \tors \Omega^q(X)
    \end{equation}
\end{lemma}
\begin{proof}
    Let $j: U \coloneq \mathbf{T}_N \hookrightarrow X$ denote the canonical open immersion of the maximal torus of $X$, which is induced by the inclusion $S_\sigma \hookrightarrow M$. By functoriality, we have the following commutative diagram,
\[\begin{tikzcd}
	{\Omega^q(U)} & {\dani^q(U)} \\
	{\Omega^q(X)} & {\dani^q(X)}
	\arrow["{\beta_U^q}", from=1-1, to=1-2]
	\arrow["{\Omega^q(j)}", from=2-1, to=1-1]
	\arrow["{\beta_X^q}", from=2-1, to=2-2]
	\arrow["{\dani^q(j)}"', from=2-2, to=1-2]
\end{tikzcd}\]
    
    Now, since $U$ is smooth, $\beta_U^q$ is an isomorphism. The injectivity of $\beta_U^q$ in the diagram implies that $\ker(\beta^q)(X) \subseteq \ker(\Omega^q(j))$. By the definition of $\dani^q(X)$ in Equation \ref{eq:exact_1}, $\dani^q(X)$ is a subgroup of $k[S_\sigma] \otimes \bigwedge^q(M)$. The map $j^*:k[S_\sigma] \to k[M]$ is injective because $k[S_\sigma]$ is torsion-free, while the map $\bigwedge^q(M) \to \bigwedge^q(M)$ corresponding to $j$ is simply the identity, so the map $k[S_\sigma] \otimes \bigwedge^q(M) \to k[M ] \otimes \bigwedge^q (M)$ is injective. Thus $\dani^q(j)$ is injective. Therefore $\ker(\beta^q)(X) = \ker(\Omega^q(j))$. Since $\Omega^q(j)$ is given by localisation at all of $k[M]$, $\ker(\Omega^q(j)) = \tors \Omega^q(j)$.
\end{proof}
\section{$K_1$-regularity and simplicial cones}\label{sec:simplicial}

\begin{definition}
For a monoid $K$ such that $K \cap (-K) = \{ O \}$, where $O$ is the neutral element of $K$, we write $a \leq b$ for $a,b \in K$ to denote the partial order
\begin{equation}
	\exists c \in K \textst a + c = b.
\end{equation}
\end{definition}

\begin{definition} \label{definition: hilbert basis}
For a subset $S$ of a monoid $K$ such that $K \cap (-K) = \{ O \}$, we denote by $\Hilb(S)$ the set of minimal elements with respect to $\leq$, also called the \emph{Hilbert basis} of $S$,
\begin{equation}
\Hilb(S) \coloneq \{x \in S \backslash \{ O \} \mid \forall y \in S, y \leq x \implies y = x \}.
\end{equation}
\end{definition}

The monoid $S_\sigma$ satisfies the condition to have the partial order $\leq$ if and only if $\sigma\dual$ is strongly convex, or equivalently that $\sigma$ is full-dimensional. For brevity, we write $H_\sigma \coloneq \Hilb(S_\sigma)$ for a cone $\sigma$ in the lattice $N$, and if $\theta < \sigma\dual$ is a dual face, we write $\Hilb(\theta) \coloneq \Hilb({\theta \cap M})$.

\begin{lemma} \label{lemma: hilbert basis face}
If $\sigma$ is a strongly convex cone in the lattice $N$ and $\theta < \sigma\dual$, then $\Hilb(\theta) = \theta \cap H_{\sigma }$. 
\end{lemma}
\begin{proof}
If $\theta < \sigma\dual$ then there is some dual vector $n \in \sigma$ such that $\theta = n^\perp \cap \sigma\dual$. Hence if $a, b \in \sigma\dual$ and $a + b \in \theta$ then since $\langle a,n\rangle, \langle b,n \rangle \geq 0$ we have 
\begin{align}
0 = \langle a + b, n \rangle = \langle a, n \rangle + \langle b, n \rangle \\
\implies \langle a, n \rangle = \langle b, n \rangle = 0 \implies a,b \in \theta.
\end{align} 
Thus $a \leq b \in \theta $ implies $a \in \theta $, so $\Hilb(\theta) \subseteq \theta \cap H_\sigma$. Conversely, if $a \in \theta \cap H_\sigma$ then $\forall b \in \theta \cap M \subset S_\sigma$ we have $b \leq a \implies b=a$, so $\theta \cap H_\sigma \subseteq \Hilb(\theta)$.

\end{proof}

\begin{definition}
If $\sigma$ is a full-dimensional cone in the lattice $N$ and $\theta < \sigma$, we denote by $\Rays(\theta)$ the set of ray generators, 

\begin{equation}
    \Rays(\theta) \coloneq \bigsqcup_{\eta \in \theta(1)} \Hilb(\eta),
\end{equation}

Note that for any ray, $|\Hilb(\eta)| = 1$. For brevity, we write $R_\sigma = \Rays(\sigma\dual)$, which has cardinality $|\sigma\dual(1)|$. 
\end{definition}

 By Lemma \ref{lemma: hilbert basis face}, $\Hilb(\eta) \subseteq H_\sigma$ for each $\eta \in \sigma\dual(1)$, so $R_\sigma \subseteq H_\sigma$, but it might not hold that $R_\sigma = H_\sigma$. We will prove Theorem \ref{thm:K1} using the fact that $|H_\sigma \backslash R_\sigma|$ distinguishes simplicial cones from smooth cones and is detected by $\coker(\beta_X^1)$.

\begin{lemma}\cite[Lemma 1.3.10]{cox_little_schenck} \label{lem:simplicial_iff_hilbert}
	For a simplicial lattice cone $(\sigma,N)$, the affine toric variety $U(\sigma,N)$ is smooth if and only if $H_\sigma = R_\sigma$. 
\end{lemma}

\begin{lemma} \label{lemma: coker1 = H less R}
	Let $X=U(\sigma,N)$ be an affine toric variety such that $\sigma$ is full-dimensional. If $R_{\sigma} \neq H_{\sigma}$ then $\coker(\beta^1_X) \neq 0$.
\end{lemma}
\begin{proof}
	Suppose there exists $h \in H_\sigma \backslash R_\sigma$. Then $\dim \Gamma(h)\geq 2$, so 
	\begin{equation}
		\dim \dani^1_{X}(h) = \dim \Gamma(h) \geq 2.
	\end{equation}
	On the other hand, $\Omega^1_{X}(h) = \langle d\chi^h \rangle$,  whence $\dim \dani^1_{X}(h) - \dim \Omega^1_{X}(h) \geq 1$. Therefore, since $\beta_X^1$ is an $M$-graded linear map, so $\dim \coker(\beta^1_X)(h) \geq 1$.
\end{proof}

\begin{theorem} \label{thm:K1}
    Let $X = U(\sigma, N)$ be an affine toric variety over a field of characteristic zero. If $\sigma$ is simplicial, then $X$ is smooth if, and only if, $X$ is $K_1$-regular. 
\end{theorem}
\begin{proof}
That `$X$ is smooth' $\implies$ `$X$ is $K_1$-regular' is a standard fact. We will prove the converse by assuming that $X$ is singular and $\sigma$ is simplicial, and deriving that $\coker(\beta^1_X) \neq 0$, which implies that $X$ is not $K_1$-regular by Lemma \ref{lemma: ker coker implies not K regular}. 

We can decompose $X$ as the product of the torus, $\mathbf{T}_{N(\sigma)}$, and the toric variety $X'= U(\sigma,N_{\sigma})$, so that $\sigma$ is full-dimensional in $N_\sigma$ and simplicial. By Lemma \ref{lemma: product}, it suffices to prove that $\coker(\beta_{X'}^1) \neq 0$, and so we may assume that $\sigma$ is full-dimensional in $N$ to begin with. 

Since $\sigma$ is singular and simplicial, Lemma \ref{lem:simplicial_iff_hilbert} implies that $H_{\sigma} \backslash R_{\sigma} \neq \emptyset$. Therefore $X$ satisfies the conditions of Lemma \ref{lemma: coker1 = H less R}, so $\coker(\beta^1_X)\neq 0$.
\end{proof}

\section{$K_2$-regularity and non-simplicial cones}\label{sec:non-simplicial}

In this section, we prove that non-simplicial toric varieties are not $K_2$-regular. The proof is by induction from the dimension 3 case. We will first prove the base case, then prove a combinatorial lemma needed for the induction step, and finally recombine in Lemma \ref{lemma: non-simplicial}. 

A key ingredient of the proof of the base case is a description of $\tors \Omega^1_X$, for a toric variety $X$, which equals $ \ker(\beta_X^1)$ by Lemma \ref{lemma: torsion-kernel}. Such a description has already been provided by \cite{torsion}, along with some useful notation, as follows.

\begin{definition}
For a subset $H \subset S_\sigma$, we adopt the following definitions:
\begin{itemize}
	\item Let $\ZZ^{H}$ denote the free abelian group generated by the basis $e_h \in H$. Let $\ZZ^{H+}$ denote the submonoid of non-negative integral combinations in $\ZZ^H$.
	\item Let $\phi: \ZZ^{H} \to M$ denote the homomorphism of groups given by $e_h \mapsto h$.
	\item Let us denote by $L(H)$ the subgroup $\ker(\phi) \subset \ZZ^H$.
	\item For any $\ell = \sum_{h \in H} \ell_h e_h \in \ZZ^H$, we define its support to be the set
	\begin{equation}
		\supp \ell \coloneq \{h \in H \mid \ell_h \neq 0\}.
	\end{equation}
	\item Every $\ell \in L(H)$ is uniquely representable as a difference $\ell = \ell^+ - \ell^-$ with $\ell^+, \ell^- \in \ZZ^{H+}$ and $\supp \ell^+ \cap \supp \ell^- = \emptyset$. Thus there is a well defined map of sets $L \to S_\sigma \subseteq M$ given by $\ell \mapsto \phi(\ell^+) \eqcolon \psi(\ell)$. Altmann calls $\psi (\ell)$ the ``\emph{bar-value}'' of $\ell$.
	\item For any $m \in S_\sigma$, let $H_m \coloneq \{h \in H \mid h \leq m\}$ .
	\item For any $m \in S_\sigma$, we define the group
	\begin{equation}
		T_m \coloneq \frac{L(H_m)}{\langle \ell \in L(H_m) \mid \psi(\ell) \leq m \rangle} = \frac{\langle \ell \in L(H) \mid \psi(\supp \ell) \leq m \rangle}{\langle \ell \in L(H) \mid \psi(\ell) \leq m \rangle}
	\end{equation}
\end{itemize}
\end{definition}

The main result of Altmann \cite[2.2]{torsion} is that
\begin{theorem} \label{thm:kahler_torsion}
	If $X = U(\sigma,N)$ is a toric variety then for any $m \in S_\sigma$ we have
	\begin{equation}
	\tors (\Omega_{X}^1)(m) = T_m \otimes_{\ZZ} k.
	\end{equation}
\end{theorem}

\begin{lemma} \label{lemma: base case}
    If $X = U(\sigma,N)$ is an affine toric variety such that $\sigma$ is non-simplicial, full-dimensional, and $\rank N = 3$, then either $\ker(\beta^1_X) \neq 0, \coker(\beta^1_X)\neq 0$, or $\coker(\beta^2_X)\neq 0$.
\end{lemma}
\begin{proof}
    We proceed by assuming that $\sigma$ is non-simplicial and that \begin{equation}
        \ker(\beta^1_X)  =\coker(\beta^1_X)= 0
    \end{equation} and deriving that $\coker(\beta^2_X) \neq 0$ by five small combinatorial claims. By Lemma \ref{lemma: coker1 = H less R}, since $\coker(\beta^1_X)=0$, $H_\sigma= R_\sigma$ (as in Definition \ref{definition: hilbert basis}).

    \begin{claim} \label{claim: min support for relation}
        If $\ell \in L(H_\sigma) \backslash \{O \}$, then $|\supp \ell^+|, |\supp \ell^-| \geq 2$.
    \end{claim}
    \begin{proof}
    Since $H_\sigma= R_\sigma$, if $|\supp(\ell^+)|=1$ then one extremal ray is equal to a positive linear combination of the others, which contradicts the definition of an extremal ray.
    \end{proof}

    \begin{claim} \label{claim: dim L + dim Gamma = H}
        If $m \in S_{\sigma}$ then $L((H_\sigma)_m)$ is a free abelian group and $\rank L((H_\sigma)_m) + \dim \Gamma(m)= |(H_\sigma)_m|$.
    \end{claim}
    \begin{proof}
        By definition, $L((H_\sigma)_m)$ is a subgroup of a free abelian group, so it is a free abelian group. For any $n \in \sigma \cap N$, if $\langle m,n\rangle = 0$ then for any $h < m$, since $\langle m - h, n \rangle \geq 0$, we have $\langle h,n \rangle = 0$. Conversely, if $\langle h, n \rangle = 0$ for all $h \in H_\sigma$ then $\langle m, n \rangle = 0$. Therefore $\dim \Gamma(m) = \rank \langle (H_\sigma)_m \rangle = \rank \ima \phi|_{\ZZ^{(H_\sigma)_m}}$. Therefore, by the first isomorphism theorem,
        \begin{align}
            \rank L((H_\sigma)_m) + \dim \Gamma(m) &= \rank \ker \phi|_{\ZZ^{(H_\sigma)_m}} + \rank \phi|_{\ZZ^{(H_\sigma)_m}}\\
            &= \rank \ZZ^{(H_\sigma)_m} = |(H_\sigma)_m|.
        \end{align}
    \end{proof}

    \begin{claim} \label{claim: dim Gamma = 3}
        If $\ell \in L(H_\sigma) \backslash \{ O\}$ then $\dim \Gamma(\psi(\ell)) = 3$.
    \end{claim}
    \begin{proof}
        Let $m = \psi(\ell)$. If $m \in \theta$ for some dual face $\theta < \sigma\dual$ then by Lemma \ref{lemma: hilbert basis face}, $(H_\sigma)_m \subseteq \Hilb(\theta) =  \theta \cap H_\sigma$. Since $H_\sigma = R_\sigma$, we have $(H_\sigma)_m \subseteq \theta \cap R_\sigma= \Rays(\theta)$. However, if $\dim \theta \leq 2$, then $\theta$ must be simplicial, so $|\Rays(\theta)| = \dim \theta$. Therefore $|(H_\sigma)_m| \leq \dim \theta \leq \dim \Gamma(m)$. Therefore by Claim \ref{claim: dim L + dim Gamma = H}, $\rank L((H_\sigma)_m) = 0$, which contradicts the assumption that $\ell \in L(H_\sigma)\backslash \{O \} $ and $m = \psi(\ell)$. Thus if $\ell \in L(H_\sigma) \backslash \{ O\}$ then $\dim \Gamma(\psi(\ell)) \geq 3$. Since $\rank N =3$, a cone in $M$ may be of dimension at most $3$, so $\dim \Gamma(\psi(\ell)) = 3$.  \end{proof}

    \begin{claim} \label{claim: dim L(Hm) = 1}
        For any $m \in \Hilb(\psi(L(H_\sigma)))$, $\rank L((H_\sigma)_m) = 1$.
    \end{claim}
    \begin{proof}
        Since $\ker(\beta^1_X)=0$, Theorem \ref{thm:kahler_torsion} implies that $T_m = 0$ for all $m \in S_\sigma$. Therefore
	
	\begin{equation}
		L((H_\sigma)_m) = \langle \ell \in L((H_\sigma)_m) \mid \psi(\ell) \leq m \rangle,
	\end{equation}
	
	If, moreover, $m \in \Hilb(\psi(L(H_\sigma)))$ then $\psi(\ell) \leq m$ implies $\ell \in \psi^{-1}(m)$, so
	
	\begin{equation}
		L((H_\sigma)_m) = \langle \ell \in L((H_\sigma)_m) \mid \psi(\ell) \leq m \rangle = \langle \psi^{-1}(m) \cap \ZZ^{(H_\sigma)_m}\rangle.
	\end{equation}

    Let $s$ denote $\rank \langle \phi^{-1}(m) \cap \ZZ^{(H_\sigma)_m} \rangle - 1$, and let $\{ b_0, \ldots, b_s \}$ be a basis for the sublattice $ \langle \phi^{-1}(m) \cap \ZZ^{(H_\sigma)_m} \rangle$. Then we have 
    \begin{equation}
        \langle \psi^{-1}(m) \cap \ZZ^{(H_\sigma)_m}\rangle = \langle b_i - b_j \mid 0 \leq i,j\leq s \rangle = \langle b_0 -b_i \mid 1 \leq i \leq s \rangle.
    \end{equation}
    
    Since the $b_i$ are independent for $1 \leq i \leq s$, the $(b_0-b_i)$ are also independent, so $\rank L((H_\sigma)_m) = \rank \langle b_0 - b_i \mid 1 \leq i \leq s \rangle = s$. Therefore by Claims \ref{claim: dim L + dim Gamma = H} and \ref{claim: dim Gamma = 3},
    \begin{equation} \label{eq: s + 3 = H}
        s+3=s + \dim \Gamma(m) = |(H_\sigma)_m|.
    \end{equation}

    Next, if there exists $i \neq j$ and $h \in \supp(b_i) \cap \supp(b_j)$ then $e_h - b_i ,e_h - b_j  \in \ZZ^{(H_\sigma)_m+}$, so $\psi(e_h - b_i -(e_h - b_j ))=m-h < m$, which contradicts $m \in \Hilb(\psi(L(H_\sigma)))$. Therefore $\supp(b_i) \cap \supp(b_j) = \emptyset$ for all $i \neq j$, so
    \begin{equation}
        |(H_\sigma)_m| = \sum^s_{i=0} |\supp b_i|.
    \end{equation}
    By Claim \ref{claim: min support for relation}, for each $i \neq j$, $|\supp(b_i - b_j)| \geq 4$. Therefore we have
    \begin{align}
        2|(H_\sigma)_m| &= 2 \sum_{i=0}^s|\supp b_i| \\
        &= |\supp(b_{s} - b_0)| + \sum_{i=0}^{s-1} | \supp (b_i - b_{i+1})|\\
        & \geq 4(s+1).
    \end{align}

    Since $s+3 = |(H_\sigma)_m|$, this implies $2(s+3) \geq 4(s+1)$, so $s \leq 1$. But $s \geq 1$, so $\rank L((H_\sigma)_m)=s=1$.
    \end{proof}

    \begin{claim} \label{claim: coker 2 non zero}
        If $m \in S_\sigma$, $\dim \Gamma(m) = 3$ and $\rank L((H_\sigma)_m)=1$, then 
        \begin{equation}
            \coker(\beta_X^2)(m) \neq 0
        \end{equation}
    \end{claim}
    \begin{proof}
        Since $\rank L((H_\sigma)_m)=1$, there is a unique (up to sign) generator of $L((H_\sigma)_m)$, $\ell$. By Claim \ref{claim: dim L + dim Gamma = H}, $|(H_\sigma)_m|= \dim \Gamma(m) + \rank L((H_\sigma)_m) = 3 + 1 = 4$. By Claim \ref{claim: min support for relation}, $|\supp\ell^+|, |\supp \ell^-| \geq 2$, so
        \begin{equation}
            2 + 2 \leq |\supp\ell^+|+ |\supp\ell^-| = |\supp\ell|  = |(H_\sigma)_m| = 4.
        \end{equation}
        So we must have $|\supp\ell^+| = |\supp \ell^-| = 2$. Let $e_1, e_2 = \supp(\ell^+)$ and $\{e_3,e_4 \} = \phi( \supp \ell^-)$. Since $\phi(e_i) \in H_\sigma$ for each $1 \leq i \leq 4$, and the $e_i$ are distinct, the $\phi(e_i)$ are also distinct. Therefore we have
        \begin{equation}
            \Omega^2_X(m) = \langle \chi^{m-\phi(e_1)-\phi(e_2)} d\chi^{\phi(e_1)}\wedge d\chi^{\phi(e_1)}, \chi^{m-\phi(e_3)-\phi(e_4)} d\chi^{\phi(e_3)}\wedge d\chi^{\phi(e_4)} \rangle,
        \end{equation}
        so $\dim \Omega^2_X(m) = 2$. On the other hand, $\dim \dani^2_X(m) = \dim \Gamma(m)=3$, so $\dim \Omega^2_X(m) < \dim \dani_X^2(m)$. Since $\beta^2_X$ is an $M$-graded linear map, $\coker(\beta^2_X)(m) \neq 0$. 
    \end{proof}

We can now complete the proof of the Lemma using claims \ref{claim: dim Gamma = 3}, \ref{claim: dim L(Hm) = 1}, and \ref{claim: coker 2 non zero}. Since $\sigma$ is not simplicial, $L(H_\sigma) \neq 0$, so there exists $m \in \Hilb(\psi(L(H_\sigma)))$. By Claim \ref{claim: dim Gamma = 3}, $\dim \Gamma(m)=3$ and by Claim \ref{claim: dim L(Hm) = 1}, $\rank L((H_\sigma)_m)=1$. Therefore the hypotheses of Claim \ref{claim: coker 2 non zero} are satisfied, so $\coker(\beta_X^2)\neq 0$.
\end{proof}

The proof of Lemma \ref{lemma: base case} relied on constraints which are unique to dimension $3$. On the other hand, the induction step relies on an elementary combinatorial fact which applies exactly when the dimension is at least $4$. The analogous fact for convex polytopes is that a simple and simplicial polytope of dimension at least $3$ must be a simplex \cite[p. 67]{polytopes}. By the correspondence between strongly convex cones and convex polytopes, it follows that a strongly convex cone of dimension at least $4$ whose faces and dual faces are all simplicial must itself be simplicial. In the interest of completeness, we include a proof for Lemma \ref{lem:simplicial cones} along the lines of \cite[p. 67]{polytopes}, translated into the language of convex cones. 

\begin{lemma} \label{lem:simplicial cones}
    Let $\sigma$ be a strongly convex, full-dimensional cone of dimension $d \geq 4$, such that if $\tau < \sigma$ then $\tau$ is simplicial, and if $\theta < \sigma\dual$ then $\theta$ is simplicial. Then $\sigma$ is simplicial.
\end{lemma}
\begin{proof}
    Let $\Delta$ denote the face poset of a cone $\sigma$ satisfying the hypothesis of the lemma, and $\Delta'$ the face poset of $\sigma\dual$. Recall from section 1 that we denote by $\Delta(r)$ the set of $r$-dimensional cones in $\Delta$, and likewise for $\Delta'$. Let $\iota: \Delta \to \Delta'$ denote the order reversing map $\tau \mapsto \tau^\perp \cap \sigma\dual$ and $\iota': \Delta' \to \Delta$ the map $\tau \mapsto \tau^\perp \cap \sigma$. Because $\sigma$ is a strongly convex, full dimensional cone, the following properties hold

    \begin{itemize}
        \item for any $\tau \in \Delta$, $\dim \iota(\tau) + \dim \tau = d$.
        \item for any $\tau, \rho \in \Delta$, $\rho < \tau$ if and only if $\iota(\rho) > \iota(\tau)$.
        \item $\iota$ and $\iota'$ are inverses.
    \end{itemize}

Note that the set $\Delta(1)$ and $\Delta(2)$ together form the vertices and edges of a graph. For a ray $\rho \in \Delta(1)$, let us denote by $A(\rho)$ the set of neighbours of $\rho$ in this graph. 

Let $\rho \in \Delta(1)$. Then $\iota(\rho) \in \Delta'(d-1)$ is simplicial, so $|\iota(\rho)(d-2)|=d-1$. And $\iota'(\iota(\rho)(d-2)) = \{\epsilon \in \Delta(2)\mid \rho < \varepsilon \}$, which is clearly in bijection with $A(\rho)$. Therefore, since $\iota'$ is a bijection, $|A(\rho)|=|\iota(\rho)(d-2)|=d-1$. 

Let $\tau = \iota'(\theta)$ for some $\theta \in \iota(\rho)(1)$. As $\theta < \iota(\rho)$, we have $\tau = \iota'(\theta) > \iota'\iota(\rho) = \rho$, and so $\rho \in \tau(1)$. Note that $\iota(\rho)$ and $\tau$ are both assumed to be simplicial. Since $\tau$ is simplicial, this implies $\tau(1) \backslash \{ \rho \}\subseteq A(\rho)$. Also since $\tau$ is simplicial, $|\tau(1)|=d-1$, so $|\tau(1) \backslash \{ \rho \}|=d-2$. Then, since $|A(\rho) |=d-1$, there is a unique $\rho' \in A(\rho) \backslash \tau(1)$. Conversely, since $\iota(\rho)$ is simplicial, $|\iota(\rho)(1)|=d-1$, so there are $d-1$ unique faces $\tau \in \Delta(d-1)$ such that $\rho < \tau$. Therefore, for each $\rho' \in A(\rho)$, there exists a face $\tau$ such that $\tau(1) = A(\rho) \cup \{ \rho\} \backslash \{ \rho' \}$.

Let $\rho' \in A(\rho)$, then for any $\rho'' \in A(\rho) \backslash \{ \rho' \}$. Since $d \geq 4$, $|A(\rho)| \geq 3$, so there exists $\rho''' \in A(\rho) \backslash \{ \rho',\rho'' \}$. Therefore there exists $\tau \in \Delta(d-1)$ such that $\tau(1) = A(\rho) \cup \{ \rho\} \backslash \{ \rho''' \}$. In particular, $\rho', \rho'' \in \tau(1)$, so $\rho'' \in A(\rho')$. As $\rho''$ was arbitrary, $A(\rho') \subseteq A(\rho) \cup \{\rho \} \backslash \{ \rho' \}$. Since $|A(\rho')|=d-1 = |A(\rho) \backslash \{ \rho'\} \cup \{\rho\}|$, we have $A(\rho') = A(\rho) \cup \{\rho \} \backslash \{ \rho' \}$. This implies that the connected component containing $\rho$ in the graph $\Delta(1)$ is $A(\rho) \cup \{\rho \}$; since the graph is connected, $|\Delta(1)| = |A(\rho) \cup \{\rho \}|=d$. Therefore $\sigma$ is simplicial. 

\end{proof}

\begin{lemma} \label{lemma: non-simplicial}
    Let $\sigma$ be a non-simplicial strongly convex cone, and let $X = U(\sigma,N)$. Then $\ker(\beta_X^1) \neq 0$, $\coker(\beta^1_X) \neq 0$, or $\coker(\beta_X^2) \neq 0$.
\end{lemma}
\begin{proof}
    Let $d$ denote the dimension of $X$. When $d\leq 2$ the statement is vacuously true because all cones of dimension $d \leq 2$ are simplicial. The base case $d=3$ is the statement of Lemma \ref{lemma: base case}. 
    
    Let our induction hypothesis be that the statement holds for all affine toric varieties of dimension $d$, for $d\geq 3$. Let $(\sigma,N)$ be a strongly convex non-simplicial cone of dimension $d+1$ and let $X = U(\sigma,N)$. By Lemma \ref{lemma: closed lifting}, we may assume that $X$ has no torus factors. By Lemma \ref{lem:simplicial cones}, since $\dim \sigma \geq 4$, $\sigma$ has either (a) a non-simplicial maximal face $\tau < \sigma$ or (b) a non-simplicial maximal dual face $\theta < \sigma\dual$.

    \textbf{Case (a)}
    
    Let us denote by $U$ the variety $U(\tau,N_\tau)$. By the induction hypothesis, since $\tau$ is non-simplicial and $\dim \tau = d$, then either $\ker(\beta_U^1) \neq 0$, $\coker(\beta^1_U) \neq 0$, or $\coker(\beta_U^2) \neq 0$. By Lemma \ref{lemma: open lifting}, one of $\ker(\beta_X^1)$, $\coker(\beta^1_X) $, or $\coker(\beta_X^2) $ is non-zero.

    \textbf{Case (b)}
    
   Consider the orbit closure $V \coloneq V_\sigma(\rho)$, where $\rho = \theta^\perp \cap \sigma$. Note that since $\theta$ is non-simplicial, $\theta\dual$ is also non-simplicial, and $(\theta\dual, N/N_\rho)$ is the cone-lattice pair associated to $V$. By the induction hypothesis, since $(\theta\dual, N(\rho))$ is non-simplicial and $\rank N(\rho) = d$, either $\ker(\beta_V^1)\neq 0$, $\coker(\beta^1_V) \neq 0$, or $\coker(\beta_V^2) \neq 0$. Thus, by Lemma \ref{lemma: closed lifting}, either $\ker(\beta_X^1) \neq 0$, $\coker(\beta^1_X) \neq 0$, or $\coker(\beta_X^2) \neq 0$.
\end{proof}

\begin{theorem}\label{thm:K2}
    Let $X = U(\sigma, N)$ be an affine toric variety over a field of characteristic zero. If $\sigma$ is not simplicial, then $X$ is not $K_2$-regular. 
\end{theorem}
\begin{proof}
    By Lemma \ref{lemma: ker coker implies not K regular}, the result follows immediately from Lemma \ref{lemma: non-simplicial}.
\end{proof}

\begin{remark}
    If one were only to prove that if $\sigma$ is not simplicial then $X$ is not $K_n$ regular for some $n > 2$, the same induction framework can be used, but the base case may be significantly easier, as follows.
    
    If $\sigma$ is non-simplicial and $\dim \sigma = 3$, then $\sigma\dual$ has at least four rays, $m_1, m_2, m_3, m_4$. Thus if $m = m_1 + m_2 + m_3 + m_4$ then $d\chi^{m_1} \wedge d\chi^{m_2} \wedge d\chi^{m_3} \wedge d\chi^{m_4} \in \Omega^4_X(m)$, but $\dim \dani^4_X(m) = \dim \bigwedge^4 V_{\Gamma(m)} = 0$. Therefore $\ker(\beta^4_X) \neq 0$, whence $NK_5(X) \neq 0$.
\end{remark}


\bibliographystyle{abbrv}
\bibliography{biblio}

@article {Bass_NK_groups,
    AUTHOR = {Corti\~nas, G. and Haesemeyer, C. and Walker, Mark E. and
              Weibel, C.},
     TITLE = {Bass' {$NK$} groups and cdh-fibrant {H}ochschild homology},
   JOURNAL = {Invent. Math.},
  FJOURNAL = {Inventiones Mathematicae},
    VOLUME = {181},
      YEAR = {2010},
    NUMBER = {2},
     PAGES = {421--448},
      ISSN = {0020-9910,1432-1297},
   MRCLASS = {19D35 (13D15 16E40 19D55)},
  MRNUMBER = {2657430},
MRREVIEWER = {Barry\ H.\ Dayton},
       DOI = {10.1007/s00222-010-0253-z},
       URL = {https://doi.org/10.1007/s00222-010-0253-z},
}

@article {chww-toric,
    AUTHOR = {Corti\~nas, G. and Haesemeyer, C. and Walker, Mark E. and
              Weibel, C.},
     TITLE = {The {$K$}-theory of toric varieties},
   JOURNAL = {Trans. Amer. Math. Soc.},
  FJOURNAL = {Transactions of the American Mathematical Society},
    VOLUME = {361},
      YEAR = {2009},
    NUMBER = {6},
     PAGES = {3325--3341},
      ISSN = {0002-9947,1088-6850},
   MRCLASS = {19E08 (14C35 14M25 19D25 19D55)},
  MRNUMBER = {2485429},
MRREVIEWER = {Annette\ Huber},
       DOI = {10.1090/S0002-9947-08-04750-8},
       URL = {https://doi.org/10.1090/S0002-9947-08-04750-8},
}

@misc{Shen,
      title={On higher {D}u {B}ois singularities and ${K}$-regularity}, 
      author={Wanchun Shen},
      year={2025},
      eprint={2504.12402},
      archivePrefix={arXiv},
      primaryClass={math.AG},
      url={https://arxiv.org/abs/2504.12402}, 
}

@article {chw,
    AUTHOR = {Corti\~nas, G. and Haesemeyer, C. and Weibel, C.},
     TITLE = {{$K$}-regularity, {$cdh$}-fibrant {H}ochschild homology, and a
              conjecture of {V}orst},
   JOURNAL = {J. Amer. Math. Soc.},
  FJOURNAL = {Journal of the American Mathematical Society},
    VOLUME = {21},
      YEAR = {2008},
    NUMBER = {2},
     PAGES = {547--561},
      ISSN = {0894-0347,1088-6834},
   MRCLASS = {19D35 (13D03 19D55 19E08)},
  MRNUMBER = {2373359},
MRREVIEWER = {Jens\ Hornbostel},
       DOI = {10.1090/S0894-0347-07-00571-1},
       URL = {https://doi.org/10.1090/S0894-0347-07-00571-1},
}

@article {KST21,
    AUTHOR = {Kerz, Moritz and Strunk, Florian and Tamme, Georg},
     TITLE = {Towards {V}orst's conjecture in positive characteristic},
   JOURNAL = {Compos. Math.},
  FJOURNAL = {Compositio Mathematica},
    VOLUME = {157},
      YEAR = {2021},
    NUMBER = {6},
     PAGES = {1143--1171},
      ISSN = {0010-437X,1570-5846},
   MRCLASS = {19D35 (14F42 19D55)},
  MRNUMBER = {4270122},
MRREVIEWER = {Piotr\ Kraso\'n},
       DOI = {10.1112/S0010437X21007120},
       URL = {https://doi.org/10.1112/S0010437X21007120},
}

@article {Vorst1,
    AUTHOR = {Vorst, Ton},
     TITLE = {Localization of the {$K$}-theory of polynomial extensions},
      NOTE = {With an appendix by Wilberd van der Kallen},
   JOURNAL = {Math. Ann.},
  FJOURNAL = {Mathematische Annalen},
    VOLUME = {244},
      YEAR = {1979},
    NUMBER = {1},
     PAGES = {33--53},
      ISSN = {0025-5831,1432-1807},
   MRCLASS = {18F25 (13F20)},
  MRNUMBER = {550060},
MRREVIEWER = {Daniel\ R.\ Grayson},
       DOI = {10.1007/BF01420335},
       URL = {https://doi.org/10.1007/BF01420335},
}

@book {cox_little_schenck,
    AUTHOR = {Cox, David A. and Little, John B. and Schenck, Henry K.},
     TITLE = {Toric varieties},
    SERIES = {Graduate Studies in Mathematics},
    VOLUME = {124},
 PUBLISHER = {American Mathematical Society, Providence, RI},
      YEAR = {2011},
     PAGES = {xxiv+841},
      ISBN = {978-0-8218-4819-7},
   MRCLASS = {14M25 (05A15 05E45 52B12)},
  MRNUMBER = {2810322},
MRREVIEWER = {Ivan\ Arzhantsev},
       DOI = {10.1090/gsm/124},
       URL = {https://doi.org/10.1090/gsm/124},
}

@article{danilov_geometry_1978,
	title = {The geometry of toric varieties.},
	volume = {33},
	issn = {00421316},
	url = {https://search.ebscohost.com/login.aspx?direct=true&amp;AuthType=sso&amp;db=msn&amp;AN=MR495499&amp;site=ehost-live&amp;custid=s2775460},
	number = {2(200)},
	journal = {Akademiya Nauk SSSR i Moskovskoe Matematicheskoe Obshchestvo. Uspekhi Matematicheskikh Nauk},
	author = {Danilov, V. I.},
	year = {1978},
	pages = {85},
}

@book{fulton_introduction_1993,
	series = {Annals of {Mathematics} {Studies}, 131.},
	title = {Introduction to toric varieties.},
	isbn = {0-691-00049-2},
	url = {https://search.ebscohost.com/login.aspx?direct=true&amp;AuthType=sso&amp;db=msn&amp;AN=MR1234037&amp;site=ehost-live&amp;custid=s2775460},
	publisher = {Princeton University Press, Princeton, NJ},
	author = {Fulton, William},
	year = {1993},
}

@book {polytopes,
    AUTHOR = {Ziegler, G\"unter M.},
     TITLE = {Lectures on polytopes},
    SERIES = {Graduate Texts in Mathematics},
    VOLUME = {152},
 PUBLISHER = {Springer-Verlag, New York},
      YEAR = {1995},
     PAGES = {x+370},
      ISBN = {0-387-94365-X},
   MRCLASS = {52Bxx},
  MRNUMBER = {1311028},
MRREVIEWER = {Margaret\ M.\ Bayer},
       DOI = {10.1007/978-1-4613-8431-1},
       URL = {https://doi.org/10.1007/978-1-4613-8431-1},
}

@article {torsion,
    AUTHOR = {Altmann, Klaus},
     TITLE = {Torsion of differentials on toric varieties},
   JOURNAL = {Semigroup Forum},
  FJOURNAL = {Semigroup Forum},
    VOLUME = {53},
      YEAR = {1996},
    NUMBER = {1},
     PAGES = {89--97},
      ISSN = {0037-1912,1432-2137},
   MRCLASS = {14M25 (14F10 20M14)},
  MRNUMBER = {1387811},
MRREVIEWER = {Alicia\ Dickenstein},
       DOI = {10.1007/BF02574124},
       URL = {https://doi.org/10.1007/BF02574124},
}

@article{adiprasito2024triangulationscommonstellarsubdivision,
  title={All triangulations have a common stellar subdivision},
  author={Adiprasito, Karim A and Pak, Igor},
  journal={Inventiones mathematicae},
  pages={1--21},
  year={2026},
  publisher={Springer}
}

\end{document}